\documentclass[12pt,a4paper,reqno]{amsart} 

\usepackage[T1]{fontenc}       

\usepackage{amsfonts}          

\usepackage{amsmath}

\usepackage{amssymb}

\usepackage{overpic}
\usepackage{euscript} 
\usepackage{eurosym}
\usepackage{comment}
\usepackage{mathrsfs} 
\usepackage{ifthen}
\usepackage{esint} 
\usepackage{color}

\long\def\symbolfootnote[#1]#2{\begingroup%
\def\thefootnote{\fnsymbol{footnote}}\footnote[#1]{#2}\endgroup}

{\qed\vspace{5pt}}

\usepackage{amsthm}

\newtheoremstyle{lause}
{5pt}
{5pt}
{\slshape}
{\parindent}
{\bfseries}
{.}
{.5em}
{}

\theoremstyle{lause}

\newtheoremstyle{maaritelma}
{5pt}
{5pt}
{\rmfamily}
{\parindent}
{\bfseries}
{.}
{.5em}
{}

\theoremstyle{maaritelma}
\newtheoremstyle{lause}
{5pt}
{5pt}
{\slshape}
{\parindent}
{\bfseries}
{.}
{.5em}
{}

\theoremstyle{lause}
\newtheorem{theorem}{Theorem}[section]
\newtheorem{lemma}[theorem]{Lemma}

\newtheorem{corollary}[theorem]{Corollary}

\newtheoremstyle{maaritelma}
{5pt}
{5pt}
{\rmfamily}
{\parindent}
{\bfseries}
{.}
{.5em}
{}

\theoremstyle{maaritelma}

\newtheorem{example}[theorem]{Example}
\newtheorem{remark}[theorem]{Remark}

\DeclareMathOperator{\cp}{cap}

\numberwithin{equation}{section}

\begin{document}

\thispagestyle{empty}

\begin{center}

{\large{\textbf{Integral representation of balayage on locally compact spaces and its application}}}

\vspace{18pt}

\textbf{Natalia Zorii}

\vspace{18pt}

\emph{Dedicated to Professor Wies{\l}aw Ple\'{s}niak on the occasion of his 80th birthday}\vspace{8pt}

\footnotesize{\address{Institute of Mathematics, Academy of Sciences
of Ukraine, Tereshchenkivska~3, 02000, Kyiv, Ukraine\\
natalia.zorii@gmail.com }}

\end{center}

\vspace{12pt}

{\footnotesize{\textbf{Abstract.} In the theory of inner and outer balayage of positive Radon measures on a locally compact space $X$ to arbitrary $A\subset X$ with respect to suitable, quite general function kernels, developed in a series of the author's recent papers, we find conditions ensuring the validity of the integral representations. The results thereby obtained do hold and seem to be largely new even for several interesting kernels in classical and modern potential theory, which looks promising for possible applications. As an example of such applications, we analyze how the total mass of a measure varies under its balayage with respect to fractional Green kernels.}}
\symbolfootnote[0]{\quad 2010 Mathematics Subject Classification: Primary 31C15.}
\symbolfootnote[0]{\quad Key words: Radon measures on a locally compact space; energy, consistency, and maximum principles; inner and outer balayage and their integral representations.
}

\vspace{6pt}

\markboth{\emph{Natalia Zorii}} {\emph{Integral representation of balayage on locally compact spaces and its application}}

\section{General conventions and the main result}\label{sec1} This study deals with the theory of potentials on a locally compact (Hausdorff) space $X$ with respect to a kernel $\kappa$, {\it a kernel\/} being thought of as a symmetric, lower semi\-con\-tin\-uous (l.s.c.) function $\kappa:X\times X\to[0,\infty]$. To be precise, we are concerned with the theory of inner and outer balayage of positive Radon measures $\omega$ on $X$ to {\it arbitrary} sets $A\subset X$. Such a theory, generalizing the pioneering work by Cartan \cite{Ca2} on Newtonian balayage on $\mathbb R^n$, $n\geqslant3$, to suitable, quite general $\kappa$ and $\omega$ on $X$, has been established in our recent papers \cite{Z22}--\cite{Z23b} and \cite{Z24b,Z24c}; see also \cite{Z24a,Z24b} for its applications to minimum energy problems with external fields.

Our current aim is to find conditions ensuring the validity of the integral representations for inner and outer swept measures. In various settings of outer balayage theory, such a representation is well known and particularly helpful (Bliedtner and Hansen \cite[p.~255]{BH}, Doob \cite[Eq.~(5.2)]{Doob}, Landkof \cite[p.~285]{L}). As to the integral representation of inner swept measures, it was validated before only for the $\alpha$-Riesz kernels $|x-y|^{\alpha-n}$ of order $\alpha\in(0,2]$, $\alpha<n$, on $\mathbb R^n$, $n\geqslant2$ (Zorii \cite[Theorem~5.1]{Z-bal2}).

To substantiate the integral representations for inner and outer swept measures in the framework of the theory developed in \cite{Z22}--\cite{Z23b} and \cite{Z24b,Z24c}, we need, however, to impose upon $X$, $\kappa$, and $\omega$ some additional requirements. In particular, we henceforth assume the space $X$ to have a {\it countable} base of open sets. Then it is
{\it $\sigma$-com\-p\-act} (that is, representable as a countable union of compact sets \cite[Section~I.9, Definition~5]{B1}), see \cite[Section~IX.2, Corollary to Proposition~16]{B3}; and hence
negligibility is the same as local negligibility \cite[Section~IV.5, Corollary~3 to Proposition~5]{B2}, while essential integrability is the same as integrability \cite[Section~V.1, Corollary to Proposition~9]{B2}. (For the theory of measures and integration on a locally compact space we refer to Bourbaki \cite{B2} or Edwards \cite{E2}; see also Fuglede \cite{F1} for a brief survey.)

We denote by $\mathfrak M$ the linear space of all (real-valued Radon) measures $\mu$ on $X$, equipped with the {\it vague} topology of pointwise convergence on the class $C_0(X)$ of all continuous functions $\varphi:X\to\mathbb R$ of compact support ${\rm Supp}(\varphi)$, and by $\mathfrak M^+$ the cone of all positive $\mu\in\mathfrak M$, where $\mu$ is {\it positive} if and only if $\mu(\varphi)\geqslant0$ for all positive $\varphi\in C_0(X)$. Since the space $X$ is second-countable, every $\mu\in\mathfrak M$ has a {\it countable} base of vague neighborhoods (see \cite[Lemma~4.4]{Z23b}), and therefore any vaguely bounded (hence vaguely relatively compact, cf.\ Bourbaki \cite[Section~III.1, Proposition~15]{B2}) $\mathfrak B\subset\mathfrak M$ has a {\it sequence} $(\mu_j)\subset\mathfrak B$ that converges vaguely to some $\mu_0\in\mathfrak M$.

Given $\mu,\nu\in\mathfrak M$, the {\it mutual energy} and the {\it potential} are introduced by
\begin{align*}
  I(\mu,\nu)&:=\int\kappa(x,y)\,d(\mu\otimes\nu)(x,y),\\
  U^\mu(x)&:=\int\kappa(x,y)\,d\mu(y),\quad x\in X,
\end{align*}
respectively, provided the value on the right is well defined as a finite number or $\pm\infty$. For $\mu=\nu$, the mutual energy $I(\mu,\nu)$ defines the {\it energy} $I(\mu,\mu)=:I(\mu)$ of $\mu\in\mathfrak M$.

In what follows, a kernel $\kappa$ is assumed to satisfy the {\it energy principle}, or equivalently to be {\it strictly positive definite}, which means that $I(\mu)\geqslant0$ for all (signed) $\mu\in\mathfrak M$, and moreover that $I(\mu)$ equals $0$ only for the zero measure. Then all (signed) $\mu\in\mathfrak M$ of finite energy $0\leqslant I(\mu)<\infty$ form a pre-Hil\-bert space $\mathcal E$ with the inner product $\langle\mu,\nu\rangle:=I(\mu,\nu)$ and the energy norm $\|\mu\|:=\sqrt{I(\mu)}$, cf.\ \cite[Lemma~3.1.2]{F1}. The topology on $\mathcal E$ determined by means of this norm, is said to be {\it strong}.

Besides, we shall always assume that $\kappa$ satisfies the {\it consistency} principle, which means that the cone
$\mathcal E^+:=\mathcal E\cap\mathfrak M^+$ is {\it complete} in the induced strong topology, and that the strong topology on $\mathcal E^+$ is {\it finer} than the induced vague topology on $\mathcal E^+$; such a kernel is said to be {\it perfect} (Fuglede \cite{F1}). Thus any strong Cauchy sequence (net) $(\mu_j)\subset\mathcal E^+$ converges {\it both strongly and vaguely} to the same unique measure $\mu_0\in\mathcal E^+$, the strong topology on $\mathcal E$ as well as the vague topology on $\mathfrak M$ being Hausdorff.

Yet another permanent requirement upon $\kappa$ is that it satisfies  the {\it domination} and {\it Ugaheri maximum principles}, where the former means that for any $\mu\in\mathcal E^+$ and any $\nu\in\mathfrak M^+$ with $U^\mu\leqslant U^\nu$ $\mu$-a.e., the same inequality holds on all of $X$; whereas the latter means that there exists $h\in[1,\infty)$, depending on $X$ and $\kappa$ only, such that for each $\mu\in\mathcal E^+$ with $U^\mu\leqslant c_\mu$ $\mu$-a.e., where $c_\mu\in(0,\infty)$, we have $U^\mu\leqslant hc_\mu$ on all of $X$. When $h$ is specified, we speak of {\it $h$-Ugaheri's maximum principle}, and when $h=1$, $h$-Ugaheri's maximum principle is referred to as {\it Frostman's maximum principle} (Ohtsuka \cite[Section~2.1]{O}).

Along with these permanent requirements upon the kernel in question, namely
\begin{itemize}
  \item[$(\mathcal R_1)$] {\it $\kappa$ is perfect and satisfies the domination and $h$-Ugaheri maximum principles},
\end{itemize}
we shall always assume that $(\mathcal R_2)$ and $(\mathcal R_3)$ are fulfilled, where:
\begin{itemize}
\item[$(\mathcal R_2)$] {\it $\kappa(x,y)$ is continuous for $x\ne y$}.\footnote{When speaking of a continuous function, we generally understand that the values are {\it finite} real numbers.}
  \item[$(\mathcal R_3)$] {\it $\kappa(\cdot,y)\to0$ uniformly on compact subsets of $X$ when $y\to\infty_X$.} (Here, $\infty_X$ denotes the Alexandroff point of $X$, see \cite[Section~I.9.8]{B1}.)
 \end{itemize}

 \begin{remark}\label{(a)}
$(\mathcal R_1)$--$(\mathcal R_3)$ do hold, for instance, for the following kernels:
\begin{itemize}
  \item[$\checkmark$] The $\alpha$-Riesz kernels $|x-y|^{\alpha-n}$ of order $\alpha\in(0,2]$, $\alpha<n$, on $\mathbb R^n$, $n\geqslant2$ (see  \cite[Theorems~1.10, 1.15, 1.18, 1.27, 1.29]{L}).
  \item[$\checkmark$] The associated $\alpha$-Green kernels, where $\alpha\in(0,2]$ and $\alpha<n$, on an arbitrary open subset of $\mathbb R^n$, $n\geqslant2$ (see \cite[Theorems~4.6, 4.9, 4.11]{FZ}).
   \item[$\checkmark$] The ($2$-)Green kernel, associated with the Laplacian, on a planar bounded open set (see \cite[Theorem~5.1.11]{AG}, \cite[Sections~I.V.10, I.XIII.7]{Doob}, and \cite{E}).
\end{itemize}
For all those kernels, $h=1$, that is, Frostman's maximum principle actually holds.
\end{remark}

Unless explicitly stated otherwise, assume a set $A\subset X$, $A\ne X$, to be {\it arbitrary}. We denote by $\mathfrak M^+(A)$ the cone of all $\mu\in\mathfrak M^+$ {\it concentrated on}
$A$, which means that $A^c:=X\setminus A$ is $\mu$-negligible, or equivalently that $A$ is $\mu$-mea\-s\-ur\-ab\-le and $\mu=\mu|_A$, $\mu|_A$ being the trace of $\mu$ to $A$, cf.\ \cite[Section~V.5.7]{B2}. (If $A$ is closed, then $\mathfrak M^+(A)$ consists of all $\mu\in\mathfrak M^+$ with support ${\rm Supp}(\mu)\subset A$, cf.\ \cite[Section~III.2.2]{B2}.)

Define $\mathcal E^+(A):=\mathcal E\cap\mathfrak M^+(A)$, and let $\mathcal E'(A)$ stand for the closure of $\mathcal E^+(A)$ in the strong topology on $\mathcal E^+$. Being a strongly closed subcone of the strongly complete cone $\mathcal E^+$, $\mathcal E'(A)$ is likewise {\it strongly complete}. It is also worth noting that\footnote{For the {\it inner} and {\it outer} capacities of $A\subset X$, denoted by $\cp_*A$ and $\cp^*A$, respectively, we refer to \cite[Section~2.3]{F1}. If $A$ is capacitable (e.g.\ open or compact), we write $\cp A:=\cp_*A=\cp_*A$.}
\begin{equation}\label{ifff}
\mathcal E'(A)=\{0\}\iff\mathcal E^+(A)=\{0\}\iff\cp_*A=0,
\end{equation}
which can be easily seen by use of \cite[Lemma~2.3.1]{F1}.

To avoid trivialities, suppose that $\cp_*A>0$. When approximating $A$ by compact subsets $K$, we may therefore only consider $K$ with $\cp K>0$.

\subsection{Inner and outer balayage} Recall that $(\mathcal R_1)$--$(\mathcal R_3)$ are assumed to hold. The following Theorem~\ref{th-aux} is actually a very particular case of results from \cite{Z22,Z24c}.

\begin{theorem}\label{th-aux}Fix $\omega\in\mathfrak M^+$, $\omega\ne0$, such that either $I(\omega)<\infty$, or $\omega\in\mathfrak M^+(A^c)$ and $\omega(X)<\infty$. Then there exists the only measure $\omega^A\in\mathcal E'(A)$ having the property\footnote{A proposition $\mathcal P(x)$ involving a variable point $x\in X$ is said to hold {\it nearly everywhere} ({\it n.e.}) on a set $Q\subset X$ if the set $N$ of all $x\in Q$ where $\mathcal P(x)$ fails, is of inner capacity zero. Replacing here $\cp_*N=0$ by $\cp^*N=0$, we obtain the concept of {\it quasi-everywhere} ({\it q.e.}) on $Q$. See \cite[p.~153]{F1}.}
\begin{equation}\label{n.e.1}
U^{\omega^A}=U^\omega\quad\text{n.e.\ on $A$};
\end{equation}
it is said to be the inner balayage of $\omega$ to $A$. If moreover $A$ is Borel, then the same $\omega^A$ also serves as the outer balayage $\omega^{*A}$ of $\omega$ to $A$, uniquely characterized within $\mathcal E'(A)$ by means of the equality\footnote{Compare with Fuglede \cite[Theorem~4.12]{Fu5}, pertaining to quasiclosed sets $A$. (By Fuglede \cite[Definition~2.1]{F71}, $A$ is said to be {\it quasiclosed} if it can be approximated in outer capacity by closed sets. Here it is also worth noting that a quasiclosed set is not necessarily Borel, and vice versa.)\label{F}}
\begin{equation*}
U^{\omega^{*A}}=U^\omega\quad\text{q.e.\ on $A$}.
\end{equation*}
\end{theorem}

\begin{proof}
  If $I(\omega)<\infty$, this follows from \cite[Theorems~4.3, 9.4]{Z22}, while otherwise from \cite[Theorems~2.5, 2.11]{Z24c}. In fact, if $\omega\in\mathfrak M^+(A^c)$ and $\omega(X)<\infty$, then, due to $(\mathcal R_2)$ and $(\mathcal R_3)$, $\omega$ is a  bounded measure whose potential $U^\omega$ is bounded on $A$ and continuous on every compact subset of $A$; therefore,
  \cite[Theorem~2.5]{Z24c} is indeed applicable here. When dealing with the outer balayage, we also use the fact that a locally compact space is second-cou\-n\-t\-able if and only if it is metrizable and $\sigma$-com\-pact \cite[Section~IX.2, Corollary to Proposition~16]{B3}. Being thus metrizable, the (second-countable, locally compact) space $X$ is perfectly normal (see Section~2, Proposition~7 and Section~4, Proposition~2 in \cite[Chapter~IX]{B3}),\footnote{By Urysohn's theorem \cite[Section~IX.4, Theorem~1]{B3}, a Hausdorff topological space $Y$ is said to be {\it normal} if for any two disjoint closed sets $F_1,F_2\subset Y$, there exist disjoint open sets $D_1,D_2\subset Y$ such that $F_i\subset D_i$ $(i=1,2)$. Further, a normal space $Y$ is said to be {\it perfectly normal} if each closed subset of $Y$ is a countable intersection of open sets, see \cite[Exercise~7 to Section~IX.4]{B3}.} and hence \cite[Theorem~9.4]{Z22} and \cite[Theorem~2.11]{Z24c} are applicable here as well.\end{proof}

\begin{remark}Since $\cp_*A>0$ and $\omega\ne0$, neither of $\omega^A$ or $\omega^{*A}$ is zero, which is obvious from (\ref{ifff}) and Theorem~\ref{th-aux} by noting that $U^\omega(x)>0$ for all $x\in X$, the kernel $\kappa$ being strictly positive definite, hence strictly positive (see \cite[p.~150]{F1}).\end{remark}

\begin{remark}If $\mathcal E^+(A)$ is strongly closed, then (and only then) $\mathcal E'(A)=\mathcal E^+(A)$,
and hence Theorem~\ref{th-aux} remains valid with $\mathcal E^+(A)$ in place of $\mathcal E'(A)$. As shown in \cite[Theorem~2.13]{Z24a}, this occurs e.g.\ if $A$ is closed or even quasiclosed (cf.\ footnote~\ref{F}).\end{remark}

\subsection{Integral representation of balayage}In Theorem~\ref{th-main}, presenting the main result of this work, the following $(\mathcal R_4)$, along with $(\mathcal R_1)$--$(\mathcal R_3)$, is assumed to hold:
\begin{itemize}
\item[$(\mathcal R_4)$] {\it The set of all $\varphi\in C_0(X)$ representable as potentials $U^\theta$ of measures $\theta\in\mathcal E$ is dense in the space $C_0(X)$ equipped with the inductive limit topology.}\footnote{With regard to the inductive limit topology on the space $C_0(X)$, see Bourbaki \cite[Section~II.4.4]{B4} and \cite[Section~III.1.1]{B2}, cf.\ also Section~\ref{sec-prel} below.}
  \end{itemize}

\begin{theorem}\label{th-main}
For $A\subset X$ with $\overline{A}:={\rm Cl}_XA\ne X$, consider nonzero $\omega\in\mathfrak M^+(\Omega)$, $\Omega:=X\setminus\overline{A}$, such that $I(\omega)$ or $\omega(X)$ is finite. Then the inner balayage $\omega^A$, uniquely determined within $\mathcal E'(A)$ by means of {\rm(\ref{n.e.1})}, admits the integral representation
\begin{equation}\label{i}
\omega^A=\int\varepsilon_x^A\,d\omega(x),
\end{equation}
where $\varepsilon_x$ denotes the unit Dirac measure at $x\in\Omega$. If moreover $A$ is Borel, then the integral representation holds true for the outer balayage as well, i.e.
\begin{equation}\label{o}
\omega^{*A}=\int\varepsilon_x^{*A}\,d\omega(x).
\end{equation}
\end{theorem}

\begin{remark}\label{(b)} All the $(\mathcal R_1)$--$(\mathcal R_4)$ do hold, for instance, for the following kernels:
\begin{itemize}
\item[$\checkmark$] The $\alpha$-Riesz kernels $|x-y|^{\alpha-n}$ of order $\alpha\in(0,2]$, $\alpha<n$, on $\mathbb R^n$, $n\geqslant2$.
\item[$\checkmark$] The Green kernel, associated with the Laplacian, on any open $D\subset\mathbb R^n$, $n\geqslant3$.
\item[$\checkmark$] The Green kernel, associated with the Laplacian, on bounded open $D\subset\mathbb R^2$.
\item[$\checkmark$] The $\alpha$-Green kernel of order $\alpha\in(1,2)$ for the fractional Laplacian on bounded open $D\subset\mathbb R^n$, $n\geqslant2$, of class $C^{1,1}$.\footnote{An open set $D\subset\mathbb R^n$ is said to be of class $C^{1,1}$ if for every $y\in\partial_{\mathbb R^n}D$, there exist open balls $B(x,r)\subset D$ and $B(x',r)\subset D^c$, where $r>0$, that are tangent at $y$ (see \cite[p.~458]{Bogdan}).\label{F11}}
\end{itemize}

Indeed, according to Remark~\ref{(a)}, it is enough to verify $(\mathcal R_4)$. For the first of these kernels, we begin by noting that for every $\varphi\in C_0(\mathbb R^n)$, there exist a compact set $K\subset\mathbb R^n$ and a sequence $(\varphi_j)\subset C_0^\infty(\mathbb R^n)$ (obtained by regularization \cite[p.~22]{S}) such that all the $\varphi$ and $\varphi_j$ equal $0$ on $K^c$, and moreover $\varphi_j\to\varphi$ uniformly on $K$ (hence, also in the inductive limit topology on $C_0(\mathbb R^n)$, cf.\ Lemma~\ref{foot-conv} below). Since each $\varphi\in C_0^\infty(\mathbb R^n)$ can be represented as the $\alpha$-Riesz potential of a signed measure on $\mathbb R^n$ of finite energy (see \cite[Lemma~1.1]{L} and \cite[Lemma~3.2]{Z-bal2}), $(\mathcal R_4)$ indeed holds.

As to the remaining three kernels on an open set $D\subset\mathbb R^n$, $n\geqslant2$, $(\mathcal R_4)$ follows by applying \cite[p.~75, Remark]{L} if $\alpha=2$, or \cite[Eq.~(19)]{Bogdan} otherwise, to $\varphi\in C_0^\infty(D)$, and then utilizing the same approximation technique as just above.
\end{remark}

\section{Preliminaries}

\subsection{Alternative characterizations of inner and outer balayage}
Given an arbitrary $A\subset X$, we denote by $\mathfrak C_A$ the upward directed set of all compact subsets $K$ of $A$, where $K_1\leqslant K_2$ if and only if $K_1\subset K_2$. If a net $(x_K)_{K\in\mathfrak C_A}\subset Y$ converges to $x_0\in Y$, $Y$ being a topological space, then we shall indicate this fact by writing
\[x_K\to x_0\quad\text{in $Y$ as $K\uparrow A$.}\]

\begin{theorem}\label{th-char}
Under the assumptions of Theorem~{\rm\ref{th-aux}}, the inner balayage $\omega^A$ can alternatively be characterized by means of any one of the following {\rm(i)--(iv)}.
\begin{itemize}
\item[{\rm(i)}] $\omega^A$ is uniquely determined within $\mathcal E'(A)$ by the symmetry relation
\begin{equation}\label{sym}I(\omega^A,\lambda)=I(\lambda^A,\omega)\quad\text{for all $\lambda\in\mathcal E^+$},\end{equation}
where $\lambda^A$ denotes the only measure in $\mathcal E'(A)$ with $U^{\lambda^A}=U^\lambda$ n.e.\ on $A$.
\item[{\rm(ii)}] $\omega^A$ is uniquely determined within $\mathcal E'(A)$ by any one of the limit relations
\begin{align*}
  \omega^K&\to\omega^A\quad\text{strongly in $\mathcal E'(A)$ as $K\uparrow A$},\\
  \omega^K&\to\omega^A\quad\text{vaguely in $\mathcal E'(A)$ as $K\uparrow A$},\\
  U^{\omega^K}&\uparrow U^{\omega^A}\quad\text{pointwise on $X$ as $K\uparrow A$},
\end{align*}
where $\omega^K$ denotes the only measure in $\mathcal E^+(K)$ with $U^{\omega^K}=U^\omega$ n.e.\ on $K$.
\item[{\rm(iii)}] $\omega^A$ is the only measure in the class $\Gamma_{A,\omega}$ having the property
\begin{equation*}
U^{\omega^A}=\min_{\nu\in\Gamma_{A,\omega}}\,U^\nu\quad\text{on $X$},
\end{equation*}
where
\begin{equation*}
\Gamma_{A,\omega}:=\bigl\{\nu\in\mathcal E^+:\ U^\nu\geqslant U^\omega\quad\text{n.e.\ on $A$}\bigr\}.
\end{equation*}
\item[{\rm(iv)}] $\omega^A$ is the only measure in the class $\Gamma_{A,\omega}$, introduced just above, such that
\begin{equation*}
\|\omega^A\|=\min_{\nu\in\Gamma_{A,\omega}}\,\|\nu\|.
\end{equation*}
\end{itemize}

If moreover $A$ is Borel, then all this remains valid with $\omega^A$ and "n.e." replaced throughout by $\omega^{*A}$ and "q.e.", respectively.
\end{theorem}

\begin{proof}This follows from \cite[Theorem~9.4]{Z22}, \cite[Theorem~3.1]{Z23a}, \cite[Theorem~3.1]{Z23b}, and \cite[Theorems~2.5, 2.11]{Z24c} in the same manner as in the proof of Theorem~\ref{th-aux}.
\end{proof}

\begin{corollary} Under the assumptions of Theorem~{\rm\ref{th-aux}}, if moreover Frostman's maximum principle holds, then
  $\omega^A$ is of minimum total mass in the class $\Gamma_{A,\omega}$, i.e.
\begin{equation}\label{cor1}\omega^A(X)=\min_{\nu\in\Gamma_{A,\omega}}\,\nu(X).\end{equation}
\end{corollary}

\begin{proof}Indeed, (\ref{cor1}) follows at once from Theorem~\ref{th-aux}(iii) by applying Deny's principle of positivity of mass in the form stated in \cite[Theorem~2.1]{Z23a}.\end{proof}

\begin{remark}
  However, the extremal property (\ref{cor1}) cannot serve as an alternative characterization of the inner balayage, for it does not determine $\omega^A$ uniquely. Indeed, consider the $\alpha$-Riesz kernel $|x-y|^{\alpha-n}$ of order $\alpha\leqslant2$, $\alpha<n$, on $\mathbb R^n$, $n\geqslant2$, and a proper, closed subset $A$ of $\mathbb R^n$ that is {\it not $\alpha$-thin at infinity} (take, for instance, $A:=\{|x|\geqslant 1\}$).\footnote{By Kurokawa and Mizuta \cite{KM}, a set $Q\subset\mathbb R^n$ is said to be {\it inner $\alpha$-thin at infinity} if
\begin{equation*}
 \sum_{j\in\mathbb N}\,\frac{\cp_*(Q_j)}{q^{j(n-\alpha)}}<\infty,
 \end{equation*}
where $q\in(1,\infty)$ and $Q_j:=Q\cap\{x\in\mathbb R^n:\ q^j\leqslant|x|<q^{j+1}\}$. See also \cite[Section~2]{Z-bal2}.\label{f-KM}} Then for any $\omega\in\mathcal E^+(A^c)$,
\begin{equation}\label{cor2}\omega^A\ne\omega\quad\text{and}\quad\omega^A(\mathbb R^n)=\omega(\mathbb R^n),\end{equation}
the former being obvious e.g.\ from Theorem~\ref{th-aux}, whereas the latter holds true by \cite[Corollary~5.3]{Z-bal2}.
Since $\omega,\omega^A\in\Gamma_{A,\omega}$ while $\Gamma_{A,\omega}$ is convex, combining (\ref{cor1}) with (\ref{cor2}) implies that there are actually infinitely many measures of minimum total mass in $\Gamma_{A,\omega}$, for so is each measure of the form $a\omega+b\omega^A$, where $a,b\in[0,1]$ and $a+b=1$.
\end{remark}

\begin{remark}
If $\omega\in\mathcal E^+$, then Theorems~\ref{th-aux} and \ref{th-char} still hold for an arbitrary perfect kernel satisfying the domination principle (see \cite{Z22}--\cite{Z23b}). Moreover, then $\omega^A$ is the (unique) orthogonal projection of $\omega$ onto the (convex, strongly complete) cone $\mathcal E'(A)$ in the pre-Hil\-bert space $\mathcal E$ (see \cite[Theorem~4.3]{Z22}).\footnote{For the concept of orthogonal projection in a pre-Hilbert space onto a convex cone and its characteristic properties, see Edwards \cite{E2} (Theorem~1.12.3 and Proposition~1.12.4(2)).} That is, $\omega^A\in\mathcal E'(A)$ and
\[\|\omega-\omega^A\|=\min_{\nu\in\mathcal E'(A)}\,\|\omega-\nu\|.\]
In the remaining case $\omega\not\in\mathcal E^+$, see \cite[Theorem~2.7]{Z24c} for a similar alternative characterization of $\omega^A$, involving the problem of minimization the Gauss functional
\[\|\mu\|^2-2\int U^\omega\,d\mu,\] where $\mu$ ranges over a suitable subclass of $\mathcal E'(A)$.
\end{remark}

\begin{remark}\label{rem-count}Under the requirements of Theorem~\ref{th-char}, if moreover $(\mathcal R_4)$ holds, then the characteristic property of $\omega^A$, given by the symmetry relation (\ref{sym}), needs only to be verified for certain {\it countably} many $\lambda\in\mathcal E^+$, depending on $X$ and $\kappa$ only. See \cite[Theorem~1.4]{Z23b} if $\omega\in\mathcal E^+$, and \cite[Theorem~5.2]{Z24c} otherwise.\end{remark}

\subsection{On the inductive limit topology on $C_0(X)$}\label{sec-prel} By Bourbaki \cite[Section~III.1.1]{B2}, {\it the inductive limit topology} on $C_0(X)$ is the inductive limit $\mathcal T$
of the locally convex topologies of the spaces $C_0(K;X)$, where $K$ ranges over all compact subsets of $X$, while $C_0(K;X)$ is the space of all $\varphi\in C_0(X)$ with ${\rm Supp}(\varphi)\subset K$, equipped with the topology $\mathcal T_K$ of uniform convergence on $K$. Thus, by \cite[Section~II.4, Proposition~5]{B4}, $\mathcal T$ is the finest of the locally convex topologies on $C_0(X)$ for which all the canonical injections $C_0(K;X)\to C_0(X)$, $K\subset X$ being compact, are continuous.

Since the space $X$ is $\sigma$-compact, there is a sequence of relatively compact, open sets $U_j$ with the union $X$ and such that $\overline{U}_j\subset U_{j+1}$, see \cite[Section~I.9, Proposition~15]{B1}. As the topology induced on $C_0(\overline{U}_j;X)$ by $\mathcal T_{\overline{U}_{j+1}}$ is just $\mathcal T_{\overline{U}_j}$, the space $C_0(X)$ is then the {\it strict} inductive limit of the sequence of spaces $C_0(\overline{U}_j;X)$, cf.\ \cite[Section~II.4.6]{B4}. Hence, by \cite[Section~II.4, Proposition~9]{B4}, $C_0(X)$ is Hausdorff and complete (in $\mathcal T$).

\begin{lemma}\label{foot-conv}For any sequence $(\varphi_j)\subset C_0(X)$, {\rm(i$_1$)} and {\rm(ii$_1$)} are equivalent.
\begin{itemize}\item[{\rm(i$_1$)}] $(\varphi_j)$ converges to $0$ in the strict inductive limit topology $\mathcal T$.
\item[{\rm(ii$_1$)}]There exists a compact subset $K$ of $X$ such that ${\rm Supp}(\varphi_j)\subset K$ for all $j$, and $(\varphi_j)$ converges to $0$ uniformly on $K$.\end{itemize}
\end{lemma}

\begin{proof}
The space $X$ being $\sigma$-compact, the claim can be derived from \cite[Lemma~4.2]{Z23b}. For the sake of completeness, we shall nevertheless provide its direct proof.

For any compact $K\subset X$, the topology on the space $C_0(K;X)$ induced by $\mathcal T$ is identical with the topology $\mathcal T_K$ \cite[Section~III.1, Proposition~1(i)]{B2}, whence (ii$_1$)$\Longrightarrow$(i$_1$).

For the opposite, assume $(\varphi_k)\subset C_0(X)$ approaches $0$ in $\mathcal T$.
Since $\{\varphi_k: k\in\mathbb N\}$ is then bounded in $\mathcal T$, there is a compact set $K\subset X$ such that ${\rm Supp}(\varphi_k)\subset K$ for all $k$, see \cite[Section~III.1, Proposition~2(ii)]{B2}. (This Proposition can indeed be applied here, for a locally compact, $\sigma$-compact space is paracompact \cite[Section~I.9, Theorem~5]{B1}.) Using \cite[Section~III.1, Proposition~1(i)]{B2} once again, we therefore conclude that $(\varphi_k)$ also approaches $0$ in the (Hausdorff) topology $\mathcal T_K$, which is the remaining claim.
\end{proof}

Let $C_0^\circ(X)$ consist of all $\psi\in C_0(X)$ representable as potentials $U^\theta$ of $\theta\in\mathcal E$.

\begin{lemma}\label{l-dense}If $(\mathcal R_4)$ is fulfilled, then for any given $\varphi\in C_0(X)$, there exist a sequence $(\psi_j)\subset C_0^\circ(X)$ and a compact set $K\subset X$ such that all the $\varphi$ and $\psi_j$ are equal to $0$ on $K^c$, and moreover $\psi_j\to\varphi$ uniformly on $K$ when $j\to\infty$.\end{lemma}

\begin{proof} As $X$ is second-countable, there is a countable set $S\subset C_0(X)$ having the following
property \cite[Section~V.3.1, Lemma]{B2}: for a given $\varphi\in C_0(X)$, there exist a sequence $(g_j)\subset S$ and a function $g_0\in S^+$ such that,\footnote{For any class $\mathcal F$ of real-valued functions, $\mathcal F^+$ denotes the class of all $f\geqslant0$ from $\mathcal F$.} for any small $\delta\in(0,\infty)$,
\[|\varphi-g_j|<\delta g_0\quad\text{for all $j\geqslant j_0$}.\]
This implies that, for those $\varphi\in C_0(X)$ and $(g_j)\subset S$, there is a compact set $K_0\subset X$ such that all the $\varphi$ and $g_j$ equal $0$ on $K_0^c$, while $g_j\to\varphi$ uniformly on $K_0$. Thus, by virtue of Lemma~\ref{foot-conv}, $g_j\to\varphi$ also in the strict inductive limit topology $\mathcal T$.

Noting from $(\mathcal R_4)$ that each $g_j$ can be approximated in $\mathcal T$ by $\psi\in C_0^\circ(X)$, we conclude from the above that there exists a sequence $(\psi_j)\subset C_0^\circ(X)$ which converges to $\varphi$ in the topology $\mathcal T$. Applying Lemma~\ref{foot-conv} once again, we arrive at the claim.
\end{proof}

\section{Proof of Theorem~\ref{th-main}} Assuming that a kernel $\kappa$ on a second-countable, locally compact space $X$ satisfies $(\mathcal R_1)$--$(\mathcal R_4)$, fix an arbitrary set $A\subset X$ such that $\Omega:=X\setminus\overline{A}\ne\varnothing$, and a measure $\omega\in\mathfrak M^+(\Omega)$ such that $I(\omega)$ or $\omega(X)$ is finite. Then $\omega^A$ and $\varepsilon_x^A$, where $x\in\Omega$, do exist, and they are uniquely determined by means of either of Theorems~\ref{th-aux} or \ref{th-char}.

Defining the function $x\mapsto\xi_x$ on $X$ with values in $\mathfrak M^+$ by setting
\begin{equation}\label{xi}
\xi_x:=\left\{
\begin{array}{cl}\varepsilon_x^A &\text{\ if $x\in\Omega$},\\
0&\text{\ otherwise},\\ \end{array} \right.\end{equation}
we claim that it is {\it scalar essentially integrable for the measure $\omega$}.

By Bourbaki \cite[Section~V.3.1]{B2}, this means that for every $\varphi\in C_0(X)$, the function $x\mapsto\xi_x(\varphi)$ on $X$ is essentially $\omega$-integrable, or equivalently $\omega$-integrable, the space $X$ being $\sigma$-compact (see \cite[Section~V.1, Corollary to Proposition~9]{B2}). In view of (\ref{xi}), we are thus reduced to the $\omega$-integrability of the function $x\mapsto\varepsilon_x^A(\varphi)$ on $\Omega$.

We denote by $\mathcal E^\circ$ the set of all $\theta\in\mathcal E$ with $U^\theta\in C_0(X)$, and
by $C_0^\circ(X)$ the subset of $C_0(X)$ consisting of all $U^\theta$, $\theta$ ranging over $\mathcal E^\circ$. For any $x\in\Omega$ any any $\psi\in C_0^\circ(X)$,\footnote{As usual, $\theta^+$ and $\theta^-$ denote the positive and negative parts of $\theta$ in the Hahn--Jor\-dan decomposition, see \cite[Section~III.1, Theorem~2]{B2}.}
\begin{align}\notag\varepsilon_x^A(\psi)&=\int U^\theta(y)\,d\varepsilon_x^A(y)=\int U^{\theta^+}(y)\,d\varepsilon_x^A(y)-\int U^{\theta^-}(y)\,d\varepsilon_x^A(y)\\
{}&=\int U^{\varepsilon_x}(y)\,d(\theta^+)^A(y)-\int U^{\varepsilon_x}(y)\,d(\theta^-)^A(y)=
U^{(\theta^+)^A}(x)-U^{(\theta^-)^A}(x),\label{al}
\end{align}
and hence $x\mapsto\varepsilon_x^A(\psi)$, $x\in\Omega$, is $\omega$-measurable as the difference of two l.s.c.\ functions. (The third equality in (\ref{al}) is due to (\ref{sym}) with $\lambda:=\theta^\pm$ and $\omega:=\varepsilon_x$, where $x\in\Omega$. Here we have used the fact that $\theta\in\mathcal E$ if and only if $\theta^\pm\in\mathcal E^+$, see \cite[Section~3.1]{F1}.)

We proceed by verifying that for any $\varphi\in C_0(X)$, the function $x\mapsto\varepsilon_x^A(\varphi)$, $x\in\Omega$, is $\omega$-measurable as well. By virtue of Lemma~\ref{l-dense}, there exist a sequence $(\psi_j)\subset C_0^\circ(X)$ and a compact set $K\subset X$ such that all the $\varphi$ and $\psi_j$ equal $0$ on $K^c$, and moreover $\psi_j\to\varphi$ uniformly on $K$ when $j\to\infty$. On account of the inequalities
\[\varepsilon_x^A(X)\leqslant h\varepsilon_x(X)=h<\infty,\quad x\in\Omega,\]
$h$ appearing in $h$-Ugaheri's maximum principle (see \cite[Proposition~4.1]{Z24c}), this implies that for any $x\in\Omega$ and any arbitrarily small $\delta\in(0,\infty)$,
\[\Bigl|\int(\varphi-\psi_j)\,d\varepsilon_x^A\Bigr|\leqslant\int|\varphi-\psi_j|\,d\varepsilon_x^A\leqslant h\delta\quad\text{for all $j\geqslant j_0$}.\]
(Here we have used \cite[Section~IV.4, Proposition~2]{B2}.) Being thus the pointwise limit of the sequence of $\omega$-measurable functions $x\mapsto\varepsilon_x^A(\psi_j)$, $j\in\mathbb N$, the function $x\mapsto\varepsilon_x^A(\varphi)$, $x\in\Omega$, is likewise $\omega$-measurable (Egoroff's theorem \cite[Section~IV.5, Theorem~2]{B2}).

In view of \cite[Section~IV.5, Theorem~5]{B2}, the claimed $\omega$-integrability of the function $x\mapsto\varepsilon_x^A(\varphi)$, $x\in\Omega$, where $\varphi\in C_0(X)$, is therefore reduced to the inequality
\begin{equation}\label{int}
\int|\varepsilon_x^A(\varphi)|\,d\omega(x)<\infty.
\end{equation}

Noting that every $\varphi\in C_0(X)$ is the difference of two elements of $C_0^+(X)$, we can assume that $\varphi\geqslant0$ (cf.\ \cite[Section~IV.4, Corollary~2 to Theorem~1]{B2}). Denoting by $\gamma$ the capacitary distribution on $\mathfrak S:={\rm Supp}(\varphi)$, normalized by\footnote{Here we use the fact that the capacity of any compact set is finite by the energy principle.}
\[\gamma(X)=I(\gamma)=\cp\mathfrak S\] (see \cite[Theorem~4.1]{F1}), we have $U^\gamma\geqslant c>0$ on $\mathfrak S$, a l.s.c.\ function attaining its greatest lower bound in a compact set \cite[Section~IV.6, Theorem~3]{B1}. (This bound must be strictly positive, for a strictly positive definite kernel is necessarily strictly positive \cite[p.~150]{F1}.) This yields that there exists $q\in(0,\infty)$ such that
\[\varphi\leqslant qU^\gamma=U^{q\gamma}\quad\text{on all of $X$},\]
and therefore
\begin{align}&\int\varepsilon_x^A(\varphi)\,d\omega(x)=
\int d\omega(x)\int\varphi(y)\,d\varepsilon_x^A(y)\leqslant
\int d\omega(x)\int U^{q\gamma}(y)\,d\varepsilon_x^A(y)\notag\\
{}&=\int d\omega(x)\int U^{q\gamma^A}(y)\,d\varepsilon_x(y)=
\int U^{q\gamma^A}(x)\,d\omega(x)\leqslant\int U^{q\gamma}(x)\,d\omega(x).\label{int1}\end{align}
(Here the second equality holds true by (\ref{sym}) with $\lambda:=q\gamma\in\mathcal E^+$ and $\omega:=\varepsilon_x$, where $x\in\Omega$, while the last inequality follows from $U^{\gamma^A}\leqslant U^\gamma$ on $X$, cf.\ \cite[Theorem~4.3]{Z22}.)

Now, if $I(\omega)<\infty$, then the Cauchy--Schwarz (Bunyakovski) inequality, applied to the last integral in (\ref{int1}), gives
\[\int U^{q\gamma}(x)\,d\omega(x)=\langle q\gamma,\omega\rangle\leqslant \|q\gamma\|\cdot\|\omega\|<\infty,\]
whence (\ref{int}).
Otherwise, $\omega(X)$ must be finite, and (\ref{int1}) again results in (\ref{int}) by noting that $U^\gamma\leqslant1$ on ${\rm Supp}(\gamma)$ \cite[Theorem~4.1]{F1}, hence $\gamma$-a.e., and consequently
\[U^\gamma\leqslant h<\infty\quad\text{on all of $X$},\] by the $h$-Ugaheri maximum principle.

The function $x\mapsto\xi_x(\varphi)$ being thus $\omega$-integrable for all $\varphi\in C_0(X)$, we may according to \cite[Section~V.3.1]{B2} define the Radon measure
\begin{equation}\label{nu}
\nu:=\int\xi_x\,d\omega(x)=\int\varepsilon_x^A\,d\omega(x)\quad\text{on $X$}
\end{equation}
by means of the formula
\begin{equation}\label{I}
\int\varphi(z)\,d\nu(z)=\int\biggl(\int\varphi(z)\,d\varepsilon_x^A(z)\biggr)\,d\omega(x),\quad\text{where $\varphi\in C_0(X)$.}
\end{equation}
Moreover, by virtue of \cite[Section~V.3, Proposition~2(c)]{B2}, equality (\ref{I}) remains valid when $\varphi$ is replaced by any positive l.s.c.\ function on $X$. For a given $y\in X$, we apply this to the (positive l.s.c.) function $\kappa(y,z)$, $z\in X$, and thus obtain
\begin{equation}\label{repr-th1}
U^\nu(y)=\int\biggl(\int\kappa(y,z)\,d\varepsilon_x^A(z)\biggr)\,d\omega(x)=\int U^{\varepsilon_x^A}(y)\,d\omega(x).
\end{equation}

To establish (\ref{i}), we need to prove that $\nu=\omega^A$, or equivalently (cf.\ (\ref{sym}))
\[I(\nu,\lambda)=I(\omega,\lambda^A)\quad\text{for every $\lambda\in\mathcal E^+$}.\]
Using Lebesgue--Fubini's theorem \cite[Section~V.8, Theorem~1]{B2}, we conclude from (\ref{repr-th1}) that, indeed,
\begin{align*}I(\nu,\lambda)&=\int U^\nu(y)\,d\lambda(y)=\int\biggl(\int U^{\varepsilon_x^A}(y)\,d\omega(x)\biggr)\,d\lambda(y)\\
  {}&={\int\biggl(\int U^{\varepsilon_x^A}(y)\,d\lambda(y)\biggr)\,d\omega(x)
  =\int\biggl(\int U^{\varepsilon_x}(y)\,d\lambda^A(y)\biggr)\,d\omega(x)}\\
  {}&=\int\biggl(\int\kappa(x,y)\,d\omega(x)\biggr)\,d\lambda^A(y)
  =\int U^\omega(y)\,d\lambda^A(y)=I(\omega,\lambda^A),
\end{align*}
where the fourth equality is valid by virtue of (\ref{sym}) with $\omega:=\varepsilon_x$, $x\in\Omega$. Now, substituting the equality $\nu=\omega^A$ thus verified into (\ref{nu}) we get (\ref{i}) as required.

It remains to substantiate the latter part of the theorem. But for any Borel $A$, applying Theorem~\ref{th-aux} gives
\[\omega^A=\omega^{*A},\qquad\varepsilon_x^A=\varepsilon_x^{*A},\quad\text{where $x\in\Omega$},\]
which substituted into (\ref{i}) implies (\ref{o}), thereby completing the whole proof.

\section{On applications of Theorem~\ref{th-main}}

\begin{example}Fix $n\geqslant2$, where $n\in\mathbb N$, and $\alpha\in(1,2)$. On a bounded open set $D\subset\mathbb R^n$ of class $C^{1,1}$ (see footnote~\ref{F11}), consider the fractional $\alpha$-Green kernel $g_\alpha(x,y)$ of order $\alpha$ (see e.g.\ \cite[Section~5]{Fr} or \cite[Section~IV.5.20]{L}), defined by means of
\[g_\alpha(x,y)=|x-y|^{\alpha-n}-\int|x-z|^{\alpha-n}\,d(\varepsilon_y)^{D^c}_{\kappa_\alpha}(z),\quad x,y\in D,\]
where $\nu^Q_{\kappa_\alpha}$ denotes the balayage of a measure $\nu\in\mathfrak M^+(\mathbb R^n)$ onto a closed set $Q\subset\mathbb R^n$ with respect to the $\alpha$-Riesz kernel $\kappa_\alpha(z,\zeta):=|z-\zeta|^{\alpha-n}$ on $\mathbb R^n$ (see e.g.\ \cite{BH,FZ,L,Z-bal}).

Now, consider a set $F\subset D$ that is relatively closed in $D$, and a nonzero measure $\omega\in\mathfrak M^+(\mathbb R^n)$, $n\geqslant2$, concentrated on $\Delta:=D\setminus F$ and such that either the total mass $\omega(\Delta)$ or the $\alpha$-Green energy $\int g_\alpha(x,y)\,d(\omega\otimes\omega)(x,y)$ is finite. Then the balayage $\omega^F_{g_\alpha}$ of $\omega$ onto $F$ with respect to the kernel $g_\alpha$ on the locally compact space $D$ does exist, and it is uniquely determined by means of either of Theorems~\ref{th-aux} or \ref{th-char}.\footnote{Here we simply use the term "balayage", since for closed (hence Borel) $F\subset D$, the inner and the outer balayage coincide (Theorem~\ref{th-aux}).}
Moreover, by Theorem~\ref{th-main} (cf.\ also Remark~\ref{(b)}), $\omega^F_{g_\alpha}$ admits the integral representation
\begin{equation}\label{FF}
\omega^F_{g_\alpha}=\int(\varepsilon_x)^F_{g_\alpha}\,d\omega(x).
\end{equation}

\begin{theorem}
Under the stated assumptions,\footnote{Since the $\alpha$-Green kernel satisfies Frostman's maximum principle (see Remark~\ref{(a)} above), we observe from \cite[Proposition~4.1]{Z24c} that, in general, $\omega^F_{g_\alpha}(F)\leqslant\omega(\Delta)$.}
\begin{equation*}
\omega^F_{g_\alpha}(F)<\omega(\Delta).
\end{equation*}
\end{theorem}

\begin{proof}
 As seen from (\ref{FF}), the theorem will be established once we show that
 \begin{equation*}
(\varepsilon_x)^F_{g_\alpha}(F)<1.
\end{equation*}
Since, by virtue of Frostman \cite[Eq.~(7)]{Fr},
\[(\varepsilon_x)^F_{g_\alpha}=(\varepsilon_x)^{F\cup D^c}_{\kappa_\alpha}\bigl|_F,\]
this is reduced, in turn, to the inequality
\begin{equation*}
(\varepsilon_x)^{F\cup D^c}_{\kappa_\alpha}(F)<1.
\end{equation*}

But the set $F\cup D^c$ is not, obviously, $\alpha$-thin at infinity (see footnote~\ref{f-KM}), whence  
\begin{equation*}
(\varepsilon_x)^{F\cup D^c}_{\kappa_\alpha}(F\cup D^c)=\varepsilon_x(\mathbb R^n)=1
\end{equation*}
according to \cite[Corollary~5.3]{Z-bal2}. It thus remains to verify that
\[(\varepsilon_x)^{F\cup D^c}_{\kappa_\alpha}(D^c)>0,\]
which however follows at once from the description of the support of the $\alpha$-Riesz swept measure $(\varepsilon_x)^{F\cup D^c}_{\kappa_\alpha}$, given in \cite[Theorem~8.5]{Z-bal} (cf.\ also Bogdan \cite[Lemma~6]{KB0}).
\end{proof}
\end{example}

\section{Acknowledgements} This research was supported in part by a grant from the Simons Foundation (1030291, N.V.Z.). The author also thanks B\'{e}la Nagy and Szil\'{a}rd R\'{e}v\'{e}sz for useful discussions on the content of the paper.

\end{document}